\documentclass[a4paper]{article}

\usepackage[utf8]{inputenc}
\usepackage[T1]{fontenc}

\usepackage{amsmath}
\usepackage{amssymb}
\usepackage{hyperref}

\usepackage{amssymb}
\usepackage{amsmath}
\usepackage{mathtools}
\usepackage{mathrsfs}
\usepackage{textcomp}
\usepackage{hhline}
\usepackage{diagbox}
\usepackage{mathdots,enumerate}
\usepackage{graphicx}
\usepackage{tikz}
\usepackage{graphicx}      

\newcommand{\N}{\mathbb{N}}
\newcommand{\C}{\mathbb{C}}

\newcommand{\Ac}{\mathcal{A}}

\newcommand{\Mc}{\mathcal{M}}

\newcommand{\Rc}{\mathcal{R}}
\newcommand{\Lc}{\mathcal{L}}

\newcommand{\Rf}{\mathfrak{R}}
\newcommand{\dom}{{\rm dom\,}}
\newcommand{\mul}{{\rm mul\,}}

\newcommand{\ran}{{\rm ran\,}}

\newcommand{\Span}{{\rm span\,}}

\newcommand{\codim}{{\rm codim\,}}

\DeclareMathOperator{\gr}{graph}

\newenvironment{smallpmatrix}
{\left(\begin{smallmatrix}}
{\end{smallmatrix}\right)}

\newenvironment{smallbmatrix}
{\left[\begin{smallmatrix}}
{\end{smallmatrix}\right]}

\usepackage{pgf,tikz,pgfplots}
\pgfplotsset{compat=1.14}
\usepackage{mathrsfs}
\usetikzlibrary{arrows}

\usepackage[amsmath,amsthm,thmmarks]{ntheorem}

\theoremstyle{plain}
\newtheorem{defi}{Definition}[section]
\newtheorem{theorem}[defi]{Theorem}
\newtheorem{prop}[defi]{Proposition}
\newtheorem{lemma}[defi]{Lemma}
\newtheorem{cor}[defi]{Corollary}
\newtheorem{ex}[defi]{Example}

\title{The spectrum and the Weyr characteristics of operator pencils and linear relations
 \thanks{H.G. acknowledges the support by the Deutsche Forschungsgemeinschaft (DFG) within the Research Training Group GRK 2583 "Modeling, Simulation and Optimization of Fluid Dynamic Applications”.}
        }

\author{Hannes Gernandt\thanks{Institut für Mechanik und Meerestechnik, TU Hamburg, Ei\ss endorfer Stra\ss e 42, 21073 Hamburg, Germany ({\tt hannes.gernandt@tuhh.de}).}
        \and Carsten Trunk \thanks{Institute of Mathematics, TU Ilmenau, Weimarer Stra\ss e 25, 98693 Ilmenau, Germany (
{\tt carsten.trunk@tu-ilmenau.de}).}}

\begin{document}

\maketitle

\begin{abstract}
The relation between the spectra of operator pencils with unbounded coefficients and of  associated  linear relations is investigated.
It turns out that various types of spectrum coincide
and the same is true for the Weyr characteristics. This characteristic describes how many independent Jordan chains up to a certain length exist. Furthermore, the change of this characteristic subject to one-dimensional perturbations is investigated.
\end{abstract}







\section{Introduction}
In this paper we study the relation between the spectrum of operator pencils and that of linear relations.
Operator pencils are often used to model partial differential-algebraic equations
\begin{align}
\label{dae}
\tfrac{d}{dt}Ex(t)=Ax(t), \quad Ex(0)=Ex_0,
\end{align}
which can be viewed as coupled equations of partial differential equations subject to linear constraints.

Here the operators $E$ and $A$ are between Banach spaces $X,Y$ and typically $E$ is bounded and $A$ is closed and densely defined~\cite{CarrShow77,FaviYagi98,ReisTisc05}. There are many examples from mechanics and electro-dynamics where partial differential-algebraic equations
\eqref{dae} can be derived.

Qualitative properties of the solution of a classical Cauchy problem are
usually described by the spectral properties of the generator.
Similarly, the spectrum of the operator pencil
$$
\Ac(\lambda):=\lambda E-A
$$
describes qualitative properties of the solutions of \eqref{dae}.

For this reason we study in this paper the spectrum of operator pencils and investigate the change of the spectral properties under small perturbations. Here small means small in rank.
In the remainder we always restrict ourselves to the following class of operator pencils $\lambda E-A$:
\begin{equation}
	\label{torte}
	\begin{cases}
		 \text{ The spaces }X \text{ and } Y  \text{ are Banach spaces},\\
		\text{ the operator } E:X\rightarrow Y \text{  is bounded,}\\
	\text{ the operator } A:X\supset \dom A\rightarrow Y \text{ is closed and densely defined and}\\
	\text{ there exists } \mu\in\C \text{ with }
	\mu E-A \text{ maps }\dom A \text{ bijectively to } Y.
	\end{cases}
\end{equation}
These assumptions are motivated by abstract differential-algebraic equations see e.g.\ \cite{CarrShow77,FaviYagi98,ReisTisc05}.
The \emph{spectrum} $\sigma(\Ac)$ of an operator pencil $\Ac(\lambda):=\lambda E-A$  is the complement of the set of all
\text{resolvent points}  $\rho(\Ac)$,
$$
\sigma(\Ac) :=
(\C\cup\{\infty\})\setminus\rho(\Ac),
$$
where  $\mu\in\C$ is in $\rho(\Ac)$
if $(\mu E-A)^{-1}$ exists and is a bounded, everywhere defined operator. Furthermore, $\infty$ is a resolvent point if $E$ has a bounded and everywhere defined inverse and $A$ is bounded.
Operator pencils with resolvent points are called \emph{regular}.
In particular, if \eqref{torte} holds, then $\rho(\Ac) \neq \emptyset$, which is due
to the Closed Graph Theorem and $\Ac$ is a regular operator pencil.
It is easy to see  that  the \emph{resolvent set} $\rho(\Ac)$ is open
(with respect to the topology of the compactification of $\mathbb C$).
A complex number $\lambda$ is in the \emph{point spectrum} $\sigma_p(\Ac)$ if $\ker(\lambda E-A)\neq\{0\}$ and $\infty\in\sigma_p(\Ac)$ if $\ker E\cap\dom A\neq\{0\}$.
We say that $\lambda\in\C$ is in the \emph{essential spectrum} $\sigma_{e}(\Ac)$ if $\Ac(\lambda)$ is not a Fredholm operator. Furthermore, $\infty\in\sigma_{e}(\Ac)$ if $E|_{\dom A}$ is not a Fredholm operator.

In this note we tudy the spectrum of operator pencils subject to two types of perturbations:
For $u,w\in Y$ and $v',w'$ in the dual space $X'$
\begin{align}
\label{rk_one_intro}
\hat \Ac(\lambda)=\Ac(\lambda)+\lambda uv'-uw'\quad \mbox{and} \quad
\hat \Ac(\lambda)=\Ac(\lambda)+\lambda uv'-wv'.
\end{align}
The above perturbations are small in the sense that they have either a one-dimensional range or a kernel of co-dimension one. In finite dimensional spaces i.e.\ $X=\C^n$ and $Y=\C^n$ the above perturbations are called \emph{rank-one perturbations} and  $\Ac(\lambda)=\lambda E-A$ with $E,A\in\C^{n\times n}$ is called \emph{matrix pencil}. In this case, the change of spectral properties pencils, is well-investigated \cite{Bara20,BaraDodi20,BaraRoca19,Batz15,DopMordTe2008,GernTrunk17}. Recently, the change of the spectrum of self-adjoint operators in Hilbert spaces under perturbations \eqref{rk_one_intro} with $v'=0$ or $u=0$ was described in \cite{DoboHryn212,DoboHryn21}.

In \cite{GernMoal20} the authors studied the change of various types of essential spectra of operator pencils under compact perturbations in the sense of linear relations which includes \eqref{rk_one_intro}. However, therein it was assumed that $X$ and $Y$ are Hilbert spaces and that $A$ is bounded.

The essential spectrum of operator pencils is also studied under the name S-essential spectrum, see \cite{AbdmAmma13,Jeri14} for related perturbation results and \cite{Jeri15} for an overview.

In this paper we focus on the change of eigenvalues of the operator pencil and its associated root subspaces $\Rf_\lambda(E,A)$. The definition of these root subspaces is based on the notion of Jordan chain for operator pencils from \cite{Keld51}. If we collect all Jordan chains up to a certain length $k\geq 1$ in the subspace $\Rf_\lambda^k(E,A)$ then we trivially have $\Rf_\lambda^k(E,A)\subseteq\Rf_\lambda^{k+1}(E,A)$. More importantly, the spectral structure of the pencil at $\lambda$ can characterized using the numbers
\[
w_k(\lambda;E,A):=\dim\frac{\Rf_\lambda^k(E,A)}{\Rf_{\lambda}^{k+1}(E,A)},\quad k\geq 1.
\]
The collection of these numbers for all $k\geq 1$ and all eigenvalues $\lambda$ is called the \emph{Weyr characteristic} of the operator pencil.
This characteristic was introduced in \cite{Weyr85} for matrices and studied by many authors \cite{AitkTurn04,MacD33,Shap99}. It was later generalized to a special class of matrix pencils in \cite{KarcKalo86}, see also \cite{Bole98}.

One of the main results of this paper is that the Weyr characteristic of the operator pencils $\Ac(\lambda)=\lambda E-A$ and $\hat \Ac(\lambda)=\lambda \hat E-\hat A$ related by \eqref{rk_one_intro} changes according to
\begin{align}
\label{weyr_one}
|w_k(\lambda;E,A)-w_k(\lambda;\hat E,\hat A)|\leq 1.
\end{align}
Similar results have been obtained previously for pencils in finite dimensions in \cite{DopMordTe2008,GernTrunk17}, for matrices in \cite{Savc03}, and for operators in \cite{BehrLebe15,HoerMell94}. The bound \eqref{weyr_one} can be used to bound the number of  eigenvalues and their multiplicities after low-rank perturbations. Another question emerging from this is the possibility of eigenvalue assignment using only a certain (structured) class of rank-one  perturbations. This problem is of high interest in the field of circuit redesign problems outlined in \cite{SommKrau12} and studied in further detail in \cite{GernKrau17,GernTrunk17}. In this context, the rank-one perturbations \eqref{rk_one_intro} correspond to capacitors which are inserted  in a given circuit with the aim of improving some of its properties.

The proof of \eqref{weyr_one} in finite dimensions are based on invariant factors and their perturbations and, hence, cannot be applied here. We follow a different approach which is based on a
recent perturbation result on the Weyr characteristic of linear relations from \cite{LebeMart2018}. Linear relations are just subspaces of  $X\times X$, which allow a rich spectral theory, see e.g.\ \cite{AJ21,BehrHassdeSn20,Cross1998}. Here
we associate with an operator pencil $\Ac(\lambda)=\lambda E-A$ the two subspaces
\begin{equation*}
\ker[A,-E|_{\dom A}] \quad \mbox{and} \quad
\ran\begin{smallbmatrix}E|_{\dom A}\\A\end{smallbmatrix},
\end{equation*}
which are called, for obvious reasons,  the \emph{kernel} and
the \emph{range representation}, see also \cite{BergTrun16}. As an important
step on the way to prove \eqref{weyr_one}, we show in Theorem~\ref{thm:eq_weyr} below
that the Weyr characteristic of the operator pencil which satisfies~\eqref{torte}
coincides with the corresponding Weyr characteristic for the range and
the Weyr characteristic for the kernel representation. This result is of independent
interest and will be very useful in a further study of spectral properties
of operator pencils with unbounded coefficients.

\section{Linear relations}
\label{sec:prelim}

Linear relations are subspaces of  $X\times X$, where $X$ is a Banach space. Here subspaces are not necessarily closed. If the subspace is closed in $X\times X$, then it is called \emph{closed linear relation}.  For an introduction to linear relations see e.g.\ \cite{AJ21,BehrHassdeSn20,Cross1998}. A frequently used
linear relation is the graph of a linear operator $A$ in $X$
defined on it's domain $\dom A\subset X$,
\begin{equation*}
    \gr A:=\{ (x,Ax) \in X\times X ~|~ x\in \dom A\}.
\end{equation*}
It is very common to simplify notation by identifying $A$ with its 
graph. That is, we understand $A$ as a linear relation 
which is in fact given by graph$\, A$. As an example,
consider the identity $I_X$ on the vector space $X$.
If the underlying space is clear from the context, we prefer
to briefly write $I$. Then with the above identification
we have
\begin{equation*}
    I =\{ (x,x) \in X\times X ~|~ x\in X\}.
\end{equation*}

Let $\Lc,\Mc$ be linear relations in $X$, then their \emph{sum}, \emph{product} and \emph{inverse} is defined as
\begin{align*}
\Lc+\Mc&:=\{(x,y_1+y_2)~|~(x,y_1)\in \Lc, (x,y_1)\in \Mc\},\\
\Mc\Lc&:=\{(x,z)~|~(x,y)\in \Lc, (y,z)\in \Mc\},\\
\Lc^{-1}&:=\{(y,x)~|~(x,y)\in \Lc\}.
\end{align*}
Note that the sum and product are associative. Furthermore, the \emph{kernel}, \emph{domain}, \emph{range} and \emph{multi-valued part} of a linear relation $\Lc$ are given by 
\begin{align*}
\ker\Lc&:=\{x\in X ~|~ (x,0)\in\Lc\},& \mul\Lc&:=\ker \Lc^{-1},\\
\dom\Lc&:=\{x\in X ~|~(x,y)\in\Lc~\text{for some $y\in X$}\}, & \ran\Lc&:=\dom \Lc^{-1}.
\end{align*}

There is a well developed spectral theory for closed linear relations, see e.g.\ \cite{BasaCher02,DijSnoo87,DijSnoo87-b}. Starting point for a spectral theory
for closed linear relations is the same as for operators: The expression $\Lc-\lambda I$ which is usually abbreviated as   $\Lc-\lambda$. With the above operations we obtain
\begin{align*}
\Lc-\lambda =\{(x,y-\lambda x) ~|~ (x,y)\in\Lc\}
\end{align*}
and obviously, together with $\Lc$, also $\Lc-\lambda$ 
and, thus,  $(\Lc-\lambda)^{-1}$, are closed linear relations.
We call $\lambda\in\C$ a \emph{resolvent point} if 
\[
\mul(\Lc-\lambda)^{-1}=\{0\},\quad \mbox{and}\quad  \dom(\Lc-\lambda)^{-1}=X,
\]
or,  equivalently,
\[
\ker(\Lc-\lambda)=\{0\},\quad \mbox{and}\quad  \ran(\Lc-\lambda)=X.
\]
In this case, $(\Lc-\lambda)^{-1}$ is a linear operator defined
on $X$ and it is closed (see above). The Closed Graph Theorem
implies that $(\Lc-\lambda)^{-1}$ is the graph of a bounded operator
or, using the above identification, $(\Lc-\lambda)^{-1}$ is a bounded
linear operator.

Furthermore, $\infty$ is called a resolvent point if 
$$
\mul\Lc=\{0\},\quad \mbox{and}\quad \dom\Lc=X.
$$
The set of all resolvent points of a closed linear relation $\Lc$ will be denoted by $\rho(\Lc)$. It is a subset of the extended complex plane
$\mathbb C \cup \{\infty\}$. Its complement in $\C\cup\{\infty\}$ is the \emph{spectrum} $\sigma(\Lc)$. Furthermore, we say that $\lambda\in\C$ is in the \emph{point spectrum} $\sigma_p(\Lc)$ if $\ker (\Lc-\lambda)\neq\{0\}$ and $\infty\in\sigma_p(\Lc)$ if $\mul\Lc\neq\{0\}$.
A complex number $\lambda$ is in the \emph{essential spectrum} $\sigma_e(\Lc)$
if $\Lc-\lambda$ is not a Fredholm relation, i.e.,
either $\Lc-\lambda$
has a non-closed range, or at least one of the two spaces
$\ker(\Lc-\lambda)$ and $\tfrac{X}{\ran(\Lc-\lambda)}$ has infinite
dimension.
Moreover, $\infty\in\sigma_e(\Lc)$ if either $\dom\Lc$ is non-closed or
at least one of the two spaces $\mul\Lc$ and $\tfrac{X}{\dom\Lc}$ has infinite dimension.

Linear relations $\Lc$ with non-empty resolvent set admit a specific 
representation as the range of a $2\times 1$ and as the  kernel 
of a $1\times 2$ operator matrix having entries built out of the identity
operator and the resolvent $(\Lc-\mu)^{-1}$ for some $\mu \in \rho(A)$.
This is the content of the following proposition which was in parts already obtained in \cite{BasaCher02}.

\begin{prop}
\label{lem:ker_is_ran}
Let $\Lc$ be a closed linear relation in a Banach space $X$ with $\mu\in\rho(\Lc)\setminus\{\infty\}$ and $\lambda\in\C$. Then 
\begin{align}
\label{ran_eq_ker}
\Lc-\lambda=\ran\begin{smallbmatrix}
(\Lc-\mu)^{-1}\\ I+(\mu-\lambda)(\Lc-\mu)^{-1}
\end{smallbmatrix}=\ker[
I+(\mu-\lambda)(\Lc-\mu)^{-1},-(\Lc-\mu)^{-1}].
\end{align}
If $\infty\in\rho(\Lc)$ then there exists a bounded operator $L:X\rightarrow X$ such that  $\Lc-\lambda=\gr(L-\lambda)$.
\end{prop}
\begin{proof} As agreed above, we identify operators
with its graphs, 
$$
(\Lc-\mu)^{-1} =\{(x,(\Lc-\mu)^{-1} x) ~|~ x\in X\}
$$
and, hence,
\begin{align}
\label{arcangel}
\Lc-\mu =\{((\Lc-\mu)^{-1} x,x) ~|~ x\in X\}.
\end{align}
This implies 
\[
\Lc-\lambda=\Lc- \mu +(\mu -\lambda)=
\{((\Lc-\mu)^{-1} x,x+(\mu -\lambda)(\Lc-\mu)^{-1} x) ~|~ x\in X\},
\]
which shows the first equation in \eqref{ran_eq_ker}. Moreover,
every element in $\Lc-\lambda$ is of the form 
$((\Lc-\mu)^{-1} x,x+(\mu -\lambda)(\Lc-\mu)^{-1} x)$ for some $x\in X$, hence it is contained in 
$\ker[
I+(\mu-\lambda)(\Lc-\mu)^{-1},-(\Lc-\mu)^{-1}]$. The opposite inclusion is clear if $\lambda=\mu$, see \eqref{arcangel}. Let in the following $\lambda\neq\mu$ and  assume that $(x,y)\in X\times X$ is in the kernel on the right-hand side of \eqref{ran_eq_ker}. Then 
\begin{align}
\label{bottom}
x+(\mu-\lambda)(\Lc-\mu)^{-1}x=(\Lc-\mu)^{-1}y.
\end{align}
Rearranging terms we find
\[
x=-(\mu-\lambda)(\Lc-\mu)^{-1}x+(\Lc-\mu)^{-1}y=(\Lc-\mu)^{-1}(y-(\mu-\lambda)x)
\]
and hence there exists $z\in X$ such that $x=(\Lc-\mu)^{-1}z$. If we plug this into \eqref{bottom} we obtain
\[
(\Lc-\mu)^{-1}(I+(\mu-\lambda)(\Lc-\mu)^{-1})z=(\Lc-\mu)^{-1}y.
\]
Therefore, for some $\hat z\in\ker (\Lc-\mu)^{-1}$ it holds
\[
y=(I+(\mu-\lambda)(\Lc-\mu)^{-1})z+\hat z=(I+(\mu-\lambda)(\Lc-\mu)^{-1})(z+\hat z).
\]
Further, $x=(\Lc-\mu)^{-1}z=(\Lc-\mu)^{-1}(z+\hat z)$. Hence, $(x,y)$ is in the set on the left-hand side in \eqref{ran_eq_ker}.
\end{proof}

\section{The spectrum of the operator pencil and its kernel and range representations}

If in \eqref{dae} $A$ and $E$ are bounded operators (or matrices)
and if, in addition, $E$ is boundedly invertible, then the multiplication
by $E^{-1}$ from the left  (from the right) leads to a standard Cauchy problem
with generator $E^{-1}A$ (resp.\ $AE^{-1}$). If $E$ is no longer invertible,
this procedure can be repeated, where the inverse and the multiplication are in the sense of linear relations. In particular,
one easily computes $E^{-1}A=\ker[A,-E]$ and $AE^{-1}=\ran\begin{smallbmatrix}E\\A\end{smallbmatrix}$.
For obvious reasons, $E^{-1}A$ is called  the \emph{kernel} and $AE^{-1}$
the \emph{range representation}, \cite{BergTrun16}.

Some extra care is needed for unbounded $A:X\supset \dom A\rightarrow Y$. For unbounded $A$
(and bounded $E$) the linear relation $\ran\begin{smallbmatrix}E\\A\end{smallbmatrix}$, or,
to be more precise, $\ran\begin{smallbmatrix}E|_{\dom A}\\A\end{smallbmatrix}$, has domain $\dom A$
and it is a subspace of $Y\times Y$.

For unbounded $A$ (and bounded $E$) the linear relation 
$\ker[A,-E]$ is a subset of $\dom A \times X$. It is the goal to relate the spectrum of $\ker[A,-E]$ to
the spectrum  of the operator pencil $\Ac(\lambda)=\lambda E-A$. In order 
to define spectrum for the expression $\ker[A,-E]$, it has to be a subset of the Cartesian product 
of the same space, that is, one has to restrict $\ker[A,-E]$ to
$\dom A\times \dom A$. Moreover, together with the graph norm of $A$,
$\dom A$ turns into a Banach space. This is the content of the following.

Let the operator pencil $\Ac(\lambda)=\lambda E-A$ 
satisfies \eqref{torte}. We consider the restriction $E|_{\dom A}$ of $E$ to $\dom A$
and its inverse,
\[
E|_{\dom A}^{-1}=\{(Ex,x)~|~ x\in \dom A\}.
\]
A short calculation reveals that
\begin{align}\nonumber
E|_{\dom A}^{-1}A&=\{(x,z)\in \dom A\times \dom A: (x,y)\in A,\quad (y,z)\in E|_{\dom A}^{-1} \}\\\label{Wolfine}
&=\{(x,z) \in \dom A\times \dom A: Ax=Ez,\quad z\in\dom A \}\\\label{WolfineII}
&=\ker[A,-E|_{\dom A}] \subset  \dom A\times \dom A\\[1ex]\nonumber
AE^{-1}&=\{(x,z)\in Y\times Y : (x,y)\in E^{-1},\quad (y,z)\in A \}\\\label{Bella}
&=\{(x,z) \in Y\times Y : x=Ey, z=Ay, \quad y\in\dom A \} \\\nonumber
&=\ran\begin{smallbmatrix}E|_{\dom A}\\A\end{smallbmatrix}\subset Y \times Y
\end{align}
and, again, the above representations are called the \emph{kernel} and \emph{range representation} of the operator pencil $\Ac(\lambda)=\lambda E-A$, respectively, see, e.g., \cite{BasaCher02}.  The subspace $E|_{\dom A}^{-1}A$ is a subspace of $\dom A\times \dom A$. In what follows, we equip $\dom A$ with the graph norm 
$$
\|x\|_A:=\|x\|+\|Ax\|, \quad x\in \dom A.
$$

The following representation of the kernel and the range representation
uses a resolvent point of the corresponding operator pencil.
It is adopted here to linear pencils, for similar considerations see, e.g., \cite[Corollary~1.10.6]{BehrHassdeSn20}. 
\begin{lemma} 
Let $\Ac(\lambda)=\lambda E-A$ satisfy \eqref{torte} and let $\mu\in\rho(\Ac)\setminus\{\infty\}$. Then for $\lambda \in \mathbb C$
\begin{align}
\label{ran_res_rep}
    AE^{-1}-\lambda&=\ran \begin{bmatrix}
    E(A-\mu E)^{-1}\\ I_Y+(\mu-\lambda) E(A-\mu E)^{-1}
    \end{bmatrix},\\[1ex]
    E|_{\dom A}^{-1}A-\lambda&=\ker\left[
    I_{\dom A}+(\mu-\lambda)(A-\mu E)^{-1}E|_{\dom A}, \ -
    (A-\mu E)^{-1}E|_{\dom A}\right]. \label{ker_res_rep}
\end{align}
Further, $AE^{-1}-\lambda$ is closed in $Y\times Y$ and the linear relation $E|_{\dom A}^{-1}A-\lambda$ is closed in $\dom A\times\dom A$ equipped with the graph norm.
\end{lemma}

\begin{proof}
The linear relation $AE^{-1}-\lambda$ is closed if and only if 
$AE^{-1}-\mu$  is closed. We have from \eqref{Bella}
\begin{align*}
 AE^{-1}-\mu&=\ran\begin{bmatrix}E|_{\dom A}\\A\end{bmatrix}-\mu
 =\{(Ex,(A-\mu E)x  : x\in\dom A \}.
 \end{align*}
 A substitution $x=(A-\mu)^{-1}z$, $z\in Y$, gives
 \begin{align}\label{RamPamPam}
 AE^{-1}-\mu &=\ran \begin{bmatrix} E(A-\mu E)^{-1}\\I_Y
 \end{bmatrix},
\end{align}
which is, by the boundedness of the linear operator 
$E(A-\mu E)^{-1}$, obviously 
a closed subspace of $Y\times Y$. Finally, \eqref{ran_res_rep} follows from 
$$
AE^{-1}-\lambda=(AE^{-1}-\mu)+(\mu-\lambda).
$$

We show \eqref{ker_res_rep}. Equation \eqref{Wolfine} gives
\begin{align}\nonumber
E|_{\dom A}^{-1}A-\mu&=\{(x,y-\mu x)\in \dom A\times \dom A: Ax=Ey\}\\\nonumber
& = \{(x,y-\mu x)\in \dom A\times \dom A: Ax-\mu Ex=E(y-\mu x)\}\\\nonumber
& = \{(x,z)\in \dom A\times \dom A: (A-\mu E)x=Ez\}\\\label{OtroTrago}
& = \{((A-\mu E)^{-1}Ez,z) : z\in \dom A\}\\\label{NickyJam}
 &=\ran \begin{bmatrix} (A-\mu E)^{-1}E\\I_{\dom A} \end{bmatrix}.
\end{align}
The linear operator $(A-\mu E)^{-1}E$ is a bounded operator
in $\dom A$ equipped with the graph norm $\|\cdot\|_A$ as
for $x\in \dom A$ it holds
\begin{align*}
\| (A-\mu E)^{-1}Ex\|_A &\leq
 \|(A-\mu E)^{-1}E\| \|x\| + \|A (A-\mu E)^{-1}E\|\|x\|\\
 &= \|(A-\mu E)^{-1}E\| \|x\| + \|E+\mu E  (A-\mu E)^{-1}E\|\|x\|\\
 &\leq M\|x\| \leq M\|x\|_A,
\end{align*}
where the constant $M$ is given by the sum of the operator
norm of the bounded operators $(A-\mu E)^{-1}E$ and $E +\mu E (A-\mu E)^{-1}E$
considered as operators with respect to the norms in $X$ and $Y$.
Hence $(A-\mu E)^{-1}E$ is a bounded operator from $\dom A$ into
$\dom A$ with respect to the graph norm.
From \eqref{NickyJam} we see that $E|_{\dom A}^{-1}A-\mu$ is a closed subspace in $\dom A\times\dom A$ equipped with the graph norm.
Moreover, it shows $\mu \in \rho(E|_{\dom A}^{-1}A)$ and we have
\begin{align*}
\left(E|_{\dom A}^{-1}A-\mu\right)^{-1}&=  \{(z,(A-\mu E)^{-1}Ez) : z\in \dom A\}
\end{align*}
and \eqref{ker_res_rep} follows from Proposition \ref{lem:ker_is_ran}.
\end{proof}

In the following proposition we relate the kernels and ranges of an operator pencil with those of its kernel and range representation.
\begin{prop}
\label{prop:ker_ran}
Let $\Ac(\lambda)=\lambda E-A$ satisfy \eqref{torte} with $\mu\in\rho(\Ac)\setminus\{\infty\}$. Then for all $\lambda\in\C$ it holds
\begin{align}
\label{kern}
\ker(E|_{\dom A}^{-1}A-\lambda)&=\ker\Ac(\lambda),&
\ker(AE^{-1}-\lambda)&=(A-\mu E)\ker \Ac(\lambda),\\
\ran(AE^{-1}-\lambda)&=\ran\Ac(\lambda), &
\ran(E|_{\dom A}^{-1}A-\lambda)&=(A-\mu E)^{-1}\ran\Ac(\lambda). \label{range}
\end{align}
Furthermore, 
\begin{align}\label{DuaLipa}
\mul(E|_{\dom A}^{-1}A)&=\ker E\cap\dom A, & \dom(E|_{\dom A}^{-1}A)&=(A-\mu E)^{-1}\ran E|_{\dom A},\\\label{Drake}
\mul(AE^{-1})&=A(\ker E\cap\dom A), & \dom(AE^{-1})&=E\dom A.
\end{align}
\end{prop}

\begin{proof}
Obviously, one has
$$
E|_{\dom A}^{-1}A-\lambda=(E|_{\dom A}^{-1}A-\mu)+(\mu-\lambda)
$$
and  \eqref{OtroTrago} implies
\begin{align}\label{Faruko}
E|_{\dom A}^{-1}A-\lambda
& = \{((A-\mu E)^{-1}Ez,z+(\mu-\lambda)((A-\mu E)^{-1}Ez) : z\in \dom A\}.
\end{align}
Thus,
\begin{align*}
\ker(E|_{\dom A}^{-1}A-\lambda)&=\ker(I_{\dom A}+(\mu-\lambda)(A-\mu E)^{-1}E|_{\dom A})\\&=\ker((A-\mu E)+(\mu-\lambda)E|_{\dom A})\\&=\ker\Ac(\lambda).
\end{align*}
Furthermore, \eqref{ran_res_rep} implies
\begin{align*}
\ran(AE^{-1}-\lambda)&=\ran(I_Y+(\mu-\lambda)E(A-\mu E)^{-1})\\&=\ran((A-\mu E)(A-\mu E)^{-1}+(\mu-\lambda)E(A-\mu E)^{-1})\\&=\ran((A-\mu E)+(\mu-\lambda)E)\\&=\ran\Ac(\lambda).
\end{align*}
We show the remaining two equations in \eqref{kern} and \eqref{range}.
Relations \eqref{Faruko} and \eqref{ran_res_rep} yield
\begin{align*}
\ran(E|_{\dom A}^{-1}A-\lambda)&=\ran(I_{\dom A}+(\mu-\lambda)(A-\mu E)^{-1}E|_{\dom A})\\&=(A-\mu E)^{-1}\ran\Ac(\lambda),\\
\ker(AE^{-1}-\lambda)&=\ker(I_Y+(\mu-\lambda)E(A-\mu E)^{-1})\\&=(A-\mu E)\ker\Ac(\lambda).
\end{align*}

The remaining statements on the domains and the multi-valued parts in \eqref{DuaLipa} and \eqref{Drake} 
follow from \eqref{Wolfine}, \eqref{Bella}, \eqref{RamPamPam}, and 
\eqref{NickyJam}.
\end{proof}

\begin{cor}
\label{cor:index}
Let $\Ac(\lambda)=\lambda E-A$ satisfy \eqref{torte}. Then for all $\lambda\in\C$ 
\begin{align*}
\dim \ker(\Ac(\lambda))=\dim \ker(AE^{-1}-\lambda)=\dim\ker (E|_{\dom A}^{-1}A-\lambda).
\end{align*}
Further, $\ran \Ac(\lambda)$ is closed if and only if $\ran (AE^{-1}-\lambda)$ is closed and this is true if and only if $\ran (E|_{\dom A}^{-1}A-\lambda)$ is closed. Furthermore,
\[
\codim \ran(\Ac(\lambda))=\codim \ran  (AE^{-1}-\lambda)=\codim \ran (E|_{\dom A}^{-1}A-\lambda).
\]
\end{cor}

Below, we describe the relation between the spectrum of operator pencils and the associated linear relations. For related results see \cite{GernMoal20, Naki16}.
\begin{prop}
\label{prop:spectrum}
Let $\Ac(\lambda)=\lambda E-A$ satisfy \eqref{torte} then
\begin{align}
\label{spek_eq}
&\sigma(\Ac)=\sigma(AE^{-1})=\sigma(E|_{\dom A}^{-1}A),\\ 
&\sigma_p(\Ac)=\sigma_p(AE^{-1})=\sigma_p(E|_{\dom A}^{-1}A),\label{point_spek}\\
&\sigma_e(\Ac)=\sigma_e(AE^{-1})=\sigma_e(E|_{\dom A}^{-1}A). \label{ess_spek}
\end{align}
\end{prop}
\begin{proof}
The assertions \eqref{spek_eq}, \eqref{point_spek} and \eqref{ess_spek} for the finite spectrum are immediate from Proposition~\ref{prop:ker_ran} and Corollary~\ref{cor:index}.

We prove \eqref{spek_eq} for the spectral point $\infty$. If $\infty\in\rho(\Ac)$ then $E$ is invertible and $A$ is bounded. Hence, $AE^{-1}=\gr (AE^{-1})$ which implies $\mul(AE^{-1})=\{0\}$ and $\dom(AE^{-1})=Y$. Thus,  $\infty\in\rho(AE^{-1})$. 
If $\infty\in\rho(AE^{-1})$ then $\dom (AE^{-1})=Y$ implies that $E$ is surjective. It is also injective. Indeed, if there exists $x\in X$ with $Ex=0$ and $x\neq 0$ then $\mul(AE^{-1})=\{0\}$ gives $Ax=0$. Hence, for all $\lambda \in \mathbb C$ it holds
\begin{equation}\label{Londra}
(\lambda E -A)x=0.
\end{equation}
Together with $\rho(\Ac)\neq \emptyset$ we conclude $x=0$. Hence, $E$ is also injective and, thus, $E^{-1}$ is a bounded operator. Since $AE^{-1}=\gr (AE^{-1})$ is the graph of a bounded operator also $A=AE^{-1}E$ is bounded and
then, by \eqref{DuaLipa}, $\infty\in\rho(E|_{\dom A}^{-1}A)$. 

If $\infty\in\rho(E|_{\dom A}^{-1}A)$ then by \eqref{DuaLipa} the operator $E|_{\dom A}$ is injective and \linebreak $(A-\mu E)^{-1}\ran E|_{\dom A}=\dom(E|_{\dom A}^{-1}A)=X$ implies that $\dom A=X$, which is then already bounded by the Closed Graph Theorem. Furthermore, $(A-\mu E)^{-1}$ is then an isomorphism between $Y$ and $X$ and  \eqref{DuaLipa} shows the surjectivity of $E$.
Hence, $E$ is invertible which means $\infty\in\rho(\Ac)$. 

We show \eqref{point_spek}.
If $\infty\in\sigma_p(\Ac)$ then there exists $x\in\ker E\cap\dom A$ with $x\neq 0$. As $Ax= 0$ implies \eqref{Londra} for all
$\lambda \in \mathbb C$, we have $Ax\neq0$
and, by \eqref{Drake}, $\infty\in\sigma_p(AE^{-1})$.

If $\infty\in\sigma_p(AE^{-1})$ then, by \eqref{Drake},
$\ker E\cap\dom A\neq \{0\}$ and \eqref{DuaLipa} implies
$\infty\in\sigma_p(E|_{\dom A}^{-1}A)$. Finally,
since $\mul(E|_{\dom A}^{-1}A)=\ker E\cap\dom A$, 
$\infty\in\sigma_p(E|_{\dom A}^{-1}A)$is equivalent to
$\infty\in\sigma_p(\Ac)$ . 

It remains to show \eqref{ess_spek}.
 If $\infty\notin\sigma_e(\Ac)$ then $E|_{\dom A}$ is Fredholm. Hence,
 by \eqref{Drake}, $\dom(AE^{-1})=E(\dom A)=\ran E|_{\dom A}$ is closed, finite co-dimensional and  $\dim\mul AE^{-1}\leq \dim(\ker E\cap\dom A)<\infty$. Therefore, $\infty\in\sigma_e(AE^{-1})$. 
 
If $\infty\notin\sigma_e(AE^{-1})$ then $\ran E|_{\dom A}=\dom AE^{-1}$ is closed with finite co-dimension and the dimension of $ \mul(AE^{-1})=A(\ker E\cap\dom A)$, cf.\ \eqref{Drake}, is finite.
As $A$ is injective on $\ker E\cap\dom A$ (see \eqref{Londra} above),
also $\ker E\cap\dom A$ is of finite dimension and 
$E|_{\dom A}$ is Fredholm. Hence, also $(A-\mu E)^{-1}E|_{\dom A}$ is Fredholm and in particular, by \eqref{DuaLipa},
$\dom(E|_{\dom A}^{-1}A)=\ran (A-\mu E)^{-1}E|_{\dom A}$ is closed and finite co-dimensional and $\mul(E|_{\dom A}^{-1}A)=\ker E|_{\dom A}$ is finite dimensional. Hence, $\infty\notin\sigma_e(E|_{\dom A}^{-1}A)$. 

Finally, if $\infty\notin\sigma_e(E|_{\dom A}^{-1}A)$ then, by
\eqref{DuaLipa}, $(A-\mu E)^{-1}E|_{\dom A}$ is Fredholm and therefore also $E|_{\dom A}$ implying that $\infty\notin\sigma_e(\Ac)$.
\end{proof}

\section{Weyr characteristic of operator pencils and linear relations}

In this section we consider the point spectrum of operator pencils and
its associated representations. 
In the study of the point spectrum of linear operators or 
matrices the Weyr characteristic is a measure for the size
of the kernels  of  $(Z-\lambda)^n$, where $Z$ is
the operator/matrix under investigation and $\lambda$ is the eigenvalue (see, e.g., \cite{Shap99,Shap15}).
In the following we introduce this characteristic also for linear relations and operator pencils, see, e.g., \cite{DijSnoo87}.
\begin{defi}
Let $\Lc$ be a linear relation in a vector space $X$. 
The root subspaces of order $k\geq 0$ at $\lambda\in\C$ are given by 
\[
\Rf_\lambda^k(\Lc):=\ker(\Lc-\lambda)^k,\quad \Rf_\lambda(\Lc):=\bigcup_{j=1}^{\infty}\Rf_\lambda^j(\Lc)
\]
and the root subspaces at $\infty$ are given by 
\[
\Rf_\infty^k(\Lc)=\ker\Lc^{-k},\quad \Rf_\infty(\Lc):=\bigcup_{j=1}^{\infty}\Rf_\infty^j(\Lc).
\]
Furthermore, the Weyr characteristic of $\Lc$ is given by 
\begin{align*}
w_k(\Lc;\lambda)&:=\dim\frac{\Rf_\lambda^k(\Lc)}{\Rf_\lambda^{k-1}(\Lc)},\quad k\geq 1.
\end{align*}
\end{defi}
Obviously, the Weyr indices are well defined since  $\Rf_\lambda^{k-1}(\Lc)\subseteq\Rf_\lambda^k(\Lc)$ for all $k\geq 1$. 
Moreover,     it holds (see Section~\ref{sec:prelim})
$$
\Rf_\infty^k(\Lc)=\mul\Lc^k \quad\mbox{and}\quad
 \Rf_\lambda^k(\Lc)=\Rf_0^k(\Lc-\lambda).
$$

We introduce the Weyr characteristic for operator pencils.
It is based on the root subspaces which are built up by Jordan chains, see \cite{Keld51} or \cite{Mark88}. We say that $(x_1,\ldots,x_k)\in(\dom A)^k$, $k\geq 1$, is a \emph{Jordan chain of length $k$ at $\lambda\in\C$} if 
\begin{align}
\label{JC_def}
(A-\lambda E)x_1=0,\quad (A-\lambda E)x_2=Ex_1,\quad \ldots,\quad (A-\lambda E)x_k=Ex_{k-1}
\end{align}
and $(x_1,\ldots,x_k)$ is a \emph{Jordan chain of length $k$ at $\infty$} if 
\begin{align}
\label{JC_infty}
Ex_1=0,\quad Ex_2=Ax_1,\quad \ldots,\quad Ex_k=Ax_{k-1}.
\end{align}
With the above notion we introduce the \emph{root subspaces} of the operator pencil $\Ac(\lambda)=\lambda E-A$ as follows
\begin{align*}
\Rf_\lambda^k(E,A):=\{x\in \dom A ~|~ \text{$x=x_k$ fulfills \eqref{JC_def}}\},\quad \Rf_\lambda(E,A):=\bigcup_{j=1}^\infty\Rf^j_\lambda(E,A),\\
\Rf_\infty^k(E,A):=\{x\in \dom A ~|~ \text{$x=x_k$ fulfills \eqref{JC_infty}}\},\quad  \Rf_\infty(E,A):=\bigcup_{j=1}^\infty\Rf^j_\infty(E,A).
\end{align*}
From \eqref{JC_def} we obtain
\begin{align}
\label{Cosculluela}
\Rf_\lambda^k(E,A)=\Rf_0^k(E,A-\lambda E).
\end{align}
Rewriting \eqref{JC_def} gives
\begin{align*}
Ax_1=\lambda Ex_1,\quad Ax_2=E(x_1+\lambda x_2),\quad \ldots,\quad Ax_k=E(x_{k-1}+\lambda x_k),
\end{align*}
which is equivalent to
\begin{align*}
(x_1, \lambda x_1),\quad (x_2,x_1+\lambda x_2),\quad \ldots,\quad (x_k, x_{k-1}+\lambda x_k) \in E|_{\dom A}^{-1}A
\end{align*}
(see \eqref{WolfineII}). This and \eqref{Cosculluela} shows
\begin{align*}
\Rf_\lambda^k(E,A)=\ker (E|_{\dom A}^{-1}A-\lambda)^k=
\Rf_0^k(E|_{\dom A}^{-1}A-\lambda)
= \Rf_\lambda^k(E|_{\dom A}^{-1}A) ,\quad k\geq 1.
\end{align*}
Similarly one shows with \eqref{JC_infty}
\begin{align*}
\Rf_\infty^k(E,A)
= \Rf_\infty^k(E|_{\dom A}^{-1}A),\quad k\geq 1.
\end{align*}
The \emph{Weyr characteristic} of the regular operator pencil  $\Ac(\lambda)=\lambda E-A$ is defined as 
\[
w_k(\Ac;\lambda):=w_k(E|_{\dom A}^{-1}A;\lambda),\quad k\geq 1.
\]
For equivalent definitions for pencils in finite dimensional spaces, see \cite{Bole98,KarcKalo86} and for matrices, see \cite{AitkTurn04,MacD33,Shap99}.

\begin{lemma}
\label{lem:root_ker_ran}
Let $\Ac(\lambda)=\lambda E-A$ satisfy \eqref{torte}. Then for all $\lambda\in\C$
\[
\Rf_{\lambda}^k(AE^{-1})=E\Rf_{\lambda}^k(E|_{\dom A}^{-1}A),\quad k\geq 1
\]
and
\[
\Rf_{\infty}^k(AE^{-1})=A\Rf_{\infty}^k(E|_{\dom A}^{-1}A),\quad k\geq 1.
\]
\end{lemma}
\begin{proof}
Assume without restriction that $\lambda=0$. If $y\in\Rf_0^k(AE^{-1})$ then there exists  $y_1,\ldots,y_{k}\in Y$ and $x_1,\ldots,x_k\in \dom A$ satisfying
\[
(Ex_1,Ax_1)=(y_1,0),\quad \ldots,\quad (Ex_k,Ax_k)=(y_{k},y_{k-1}),\quad  y=y_k.
\]
Then
\[
Ax_1=0,\quad Ax_2=Ex_1,\quad \quad \ldots,\quad  Ax_k=Ex_{k-1},\quad  y=Ex_k
\]
which is, by \eqref{WolfineII}, $x_k\in\Rf_0^k(E|_{\dom A}^{-1}A)$ and hence $y\in E\Rf_0^k(E|_{\dom A}^{-1}A)$. The converse implication is clear. If  $y\in\Rf_{\infty}(AE^{-1})$ then there exists  $y_1,\ldots,y_{k}\in Y$ and $x_1,\ldots,x_k\in \dom A$ satisfying
\[
(Ex_1,Ax_1)=(0,y_1),\quad \ldots,\quad (Ex_k,Ax_k)=(y_{k-1},y_{k}),\quad  y=y_k.
\]
Hence, 
\[
Ex_1=0,\quad Ax_1=y_1=Ex_2,\quad \ldots,\quad Ax_{k-1}=Ex_k\quad y=Ax_k
\]
which means that $x_k\in\Rf_{\infty}^k(E|_{\dom A}^{-1}A)$ and hence $y=Ax_k\in\Rf_{\infty}^k(E|_{\dom A}^{-1}A)$. The converse is clear.
\end{proof}

\begin{lemma}
Let $\Lc$ be a linear relation in a vector space $X$. For all $k\geq 1$ 
there exists a linear injection
\[
\iota_k: \frac{\Rf_{\lambda}^{k+1}(\Lc)}{\Rf_{\lambda}^{k}(\Lc)}\rightarrow \frac{\Rf_{\lambda}^{k}(\Lc)}{\Rf_{\lambda}^{k-1}(\Lc)}.
\]
In particular, $w_{k+1}(\Lc;\lambda)\leq w_{k}(\Lc;\lambda)$ for all $k\geq 1$. 
\end{lemma}
\begin{proof}
Assume without restriction that $\lambda=0$.
Let $[x_{k+1}]\in \tfrac{\Rf_{0}^{k+1}(\Lc)}{\Rf_{0}^{k}(\Lc)}$ with
$[x_{k+1}]\neq 0$. Then there exist $x_1,\ldots,x_k\in X\setminus\{0\}$
such that 
$$
(x_{k+1},x_{k}),\ldots(x_1,0)\in\Lc.
$$
Hence $x_k\in\Rf_0^k(\Lc)\setminus\Rf_0^{k-1}(\Lc)$. The mapping
$\iota_k$ is defined by 
$$
\iota_k([x_{k+1}])=[x_k]\quad \mbox{and}\quad\iota_k(0)=0.
$$
It is indeed linear and injective since $\iota([x_{k+1}])=[x_k]=0$ implies that $x_k\in\Rf_0^{k-1}(\Lc)$ and therefore $x_{k+1}\in\Rf_0^{k}(\Lc)$ and hence $[x_{k+1}]=0$.
\end{proof}

Next, we show that the Weyr characteristic of the kernel and range
representations of an operator pencil coincide. This is one the main results. For this, we introduce the notion of 
the \emph{singular subspace}. It is given by
\[
\Rc_c(\Lc):=\Rf_0(\Lc)\cap\Rf_\infty(\Lc).
\]
We recall a result from \cite[Theorem~4.3]{BergdeSn22}.
\begin{prop}
\label{prop:sing}
Let $\Lc$ be a linear relation in a vector space $X$. Then for all $\lambda,\mu\in\C\cup\{\infty\}$ with $\lambda\neq\mu$ it holds  $\Rc_c(\Lc)=\Rf_{\lambda}(\Lc)\cap\Rf_\mu(\Lc)$. In particular, if $\mu\in\rho(\Lc)$ we have $\Rf_\mu(\Lc)=\{0\}$ and $\Rc_c(\Lc)=\{0\}$.
\end{prop}

\begin{theorem}
\label{thm:eq_weyr}
Let $\Ac(\lambda)=\lambda E-A$ satisfies \eqref{torte}. Then for all $\lambda\in\C\cup\{\infty\}$ the following Weyr indices coincide for all $k\geq 1$
\[
w_k(E|_{\dom A}^{-1}A;\lambda)=w_k(\Ac;\lambda)=w_k(AE^{-1};\lambda).
\]
\end{theorem}

\begin{proof}
Again, we consider only the case $\lambda =0$. The first equation holds by definition. The assumption \eqref{torte}  yields the existence of some $\mu\in\rho(E,A)$ and hence by Proposition~\ref{prop:spectrum}, $\mu\in\rho(E|_{\dom A}^{-1}A)$. Now  Proposition~\ref{prop:sing} gives 
\[
\{0\}=\Rc_c(E|_{\dom A}^{-1}A)=\Rf_0(E|_{\dom A}^{-1}A)\cap\Rf_\infty(E|_{\dom A}^{-1}A).
\]
As $\ker E|_{\dom A}=\mul (E|_{\dom A}^{-1}A)\subseteq \Rf_\infty(E|_{\dom A}^{-1}A)$
and by \eqref{DuaLipa}
\[
 \Rf_0(E|_{\dom A}^{-1}A)\cap\ker E|_{\dom A} \subseteq
 \Rf_0(E|_{\dom A}^{-1}A)\cap\Rf_\infty(E|_{\dom A}^{-1}A) =\{0\}
\]
and hence $E$ is injective on $\Rf_0^k(E|_{\dom A}^{-1}A)$. Now, 
the second equation follows from Lemma~\ref{lem:root_ker_ran}
\begin{align*}
w_k(AE^{-1};0)&=\dim\frac{\Rf_0^k(AE^{-1})}{\Rf_0^{k-1}(AE^{-1})}=\dim\frac{E\Rf_0^k(E|_{\dom A}^{-1}A)}{E\Rf_0^{k-1}(E|_{\dom A}^{-1}A)}\\&=\dim\frac{\Rf_0^k(E|_{\dom A}^{-1}A)}{\Rf_0^{k-1}(E|_{\dom A}^{-1}A)}=w_k(E|_{\dom A}^{-1}A;0).
\end{align*}
\end{proof}

\section{Weyr indices under one-dimensional perturbations}

In this section we investigate one-dimensional perturbations of regular operator pencils based on the underlying kernel and range representations. 
We derive formulas for the change of the Weyr characteristic from a combination of Theorem~\ref{thm:eq_weyr} and a recent perturbation result for linear relations~\cite{LebeMart2018}.
The following notion of perturbations for linear relations was used in  \cite{AzizBehr09,LebeMart2018}.

\begin{defi}
\label{defi:rkone}
Let $k\in\N$. The subspaces $\Lc,\Mc\subseteq X\times X$ are called \emph{one-dimensional perturbations} of each other, if 
\begin{align*}
\dim(\Lc,\Mc):=\max\left\{\dim\frac{\Lc}{\Lc\cap\Mc},\dim\frac{\Mc}{\Lc\cap\Mc}\right\}= 1.
\end{align*}
\end{defi}
In the special case where $\Lc$ and $\Mc$ are graphs of closed operators, i.e.\ $\Lc=\gr A$, $\Mc=\gr\hat A$ and $\mu\in\rho(A)\cap\rho(\hat A)$ then Definition~\ref{defi:rkone} is equivalent to
\[
\dim\ran((A-\mu)^{-1}-(\hat A-\mu)^{-1})=1
\]
and $Ax=\hat Ax$ for all $x\in\dom A\cap\dom \hat A$.

In the following we consider two types of perturbations of the regular operator pencil $\Ac(\lambda)=\lambda E-A$. Let $u,w\in Y$ and  $v',w'\in X'$ where $X'$ denotes the space of all continuous linear functionals on $X$. Then we consider 
\begin{subequations}
\label{rk_one}
\begin{align}
\hat \Ac(\lambda)=\Ac(\lambda)+\lambda uv'-wv'.\label{rk_one_v}\\
\hat \Ac(\lambda)=\Ac(\lambda)+\lambda uv'-uw',\label{rk_one_u}
\end{align}
\end{subequations}

In the proposition below, we show that the above perturbations 
are one-dimensional perturbations of their kernel or their range representation. 
\begin{prop}
\label{prop:rankoneisonedim}
Let $\Ac(\lambda)=\lambda E-A$ satisfy \eqref{torte}. Let $u,w\in Y$, $v'\in X'$. Then
\begin{align}
\label{dimkleinereins}
\dim(AE^{-1},(A+wv')(E+uv')^{-1})\leq 1.
\end{align}
Furthermore, if $u\in Y$ and $v',w'\in X'$ then 
\begin{align}
\dim(E|_{\dom A}^{-1}A,(E|_{\dom A}+uv')^{-1}(A+uw'))\leq 1.\label{dimkerkleinereins}
\end{align}
\end{prop}

\begin{proof}
Assume that $v'\neq 0$. There exists $x_0\in \dom A$ such that 
$X=\ker v'\dotplus \Span\{x_0\}$.
 The same for $\dom A$ gives the decomposition
\begin{align}
\label{space_dec}
\dom A=(\ker v'\cap\dom A)\dotplus \Span\{x_0\}.
\end{align}
Using the definition of the range representation and \eqref{space_dec}, we find
\begin{align*}
\Lc&:=AE^{-1}=\ran\begin{smallbmatrix}E\\ A\end{smallbmatrix}=\begin{smallbmatrix}
	E\\A
\end{smallbmatrix}\left(\ker v'\cap\dom A\right)+ \Span\left\{\begin{smallpmatrix}Ex_0\\Ax_0\end{smallpmatrix}\right\},\\
\Mc&:=\ran\begin{smallbmatrix}E+uv'\\ A+wv'\end{smallbmatrix}=\begin{smallbmatrix}
E\\A
\end{smallbmatrix}\left(\ker v'\cap\dom A\right)+ \Span\left\{\begin{smallpmatrix}Ex_0+uv'(x_0)\\Ax_0+wv'(x_0)\end{smallpmatrix}\right\}.
\end{align*}
This implies $\begin{smallbmatrix}
E\\A
\end{smallbmatrix}\left(\ker v'\cap\dom A\right)\subseteq \Lc\cap\Mc$. Hence 
\begin{align} \label{LilDurk}
\dim\frac{\Lc}{\Lc\cap\Mc}&\leq 1
\quad \mbox{and} \quad
\dim\frac{\Mc}{\Lc\cap\Mc}\leq 1.
\end{align}

To prove \eqref{dimkerkleinereins} we assume that $u\neq 0$. By
\eqref{WolfineII}
\begin{align*}
\Lc&:=E|_{\dom A}^{-1}A=\ker[A,-E|_{\dom A}],\\ 
\Mc&:=(E|_{\dom A}+uv')^{-1}(A+uw') =\ker[A+uw',-E|_{\dom A}-uv'].
\end{align*}
where we used that $\dom (A+uw')=\dom A$. Since $u\neq 0$ we have
\[
\Lc\cap\Mc=\{(x,y)\in E|_{\dom A}^{-1}A~|~v'(y)=w'(x)\}.
\]
This implies \eqref{LilDurk}.
\end{proof}
In Proposition \ref{prop:rankoneisonedim} it is shown that the perturbation given
by \eqref{rk_one_v} (or \eqref{rk_one_u}) produces a one dimensional perturbation
of the corresponding range representations, cf.\ \eqref{dimkleinereins},
(resp.\ of the corresponding kernel representations, cf.\ \eqref{dimkerkleinereins}).
But a perturbation given
by \eqref{rk_one_v} (or \eqref{rk_one_u}) can produce a two dimensional perturbation
if one considers the corresponding kernel representation 
(resp.\  the corresponding range representation). This is the content of the following example.

\begin{ex}
Consider $X=Y=\C^2$ and $E=A=\begin{smallbmatrix}0&0\\0&0
\end{smallbmatrix}\in\C^{2\times 2}$, $u=v=\begin{smallpmatrix}
1\\0
\end{smallpmatrix}$ and $w=\begin{smallpmatrix}
0\\1
\end{smallpmatrix}$. Furthermore, we write $v^*,w^*$ for the functionals $v', w'$ which is the complex-conjugate vector  transposed. Then 
\begin{align*}
    (A+uw^*)(E+uv^*)^{-1}=\ran\begin{smallbmatrix}
    1&0\\0&0\\
0&1\\0&0    
    \end{smallbmatrix}=\Span\left\{\begin{smallpmatrix}
    1\\0\\0\\0
    \end{smallpmatrix},\begin{smallpmatrix}
    0\\0\\1\\0
    \end{smallpmatrix}\right\},
\end{align*}
which is a two-dimensional subspace, whereas $AE^{-1}$ has dimension zero, i.e.\ \eqref{dimkleinereins} does not hold for arbitrary rank-one perturbations. Similarly, 
\begin{align*}
    (E+uv^*)^{-1}(A+wv^*)=\ker\begin{smallbmatrix}
    0&0&-1&0\\ 1&0&0&0
    \end{smallbmatrix}
    =\Span\left\{\begin{smallpmatrix}
    0\\1\\0\\0
    \end{smallpmatrix},\begin{smallpmatrix}
    0\\0\\0\\1
    \end{smallpmatrix}\right\},
\end{align*}
which is a two-dimensional perturbation of $E|_{\dom A}^{-1}A=\C^4$.
\end{ex}

In the proposition below, we recall bounds for the Weyr characteristic of linear relations with 
trivial singular subspace under one-dimensional perturbations~\cite[Corollary~4.6]{LebeMart2018}. 
\begin{prop}
\label{prop:francisco}
Let $\Lc$ and $\Mc$ be linear relations in a vector space $X$ with $\Rc_c(\Lc)=\Rc_c(\Mc)=\{0\}$, $\dim(\Lc,\Mc)\leq 1$ and Weyr characteristics $w(\Lc;\lambda)=(w_i(\Lc;\lambda))_{i\geq 1}$ and $w(\Mc;\lambda)=(w_i(\Mc;\lambda))_{i\geq 1}$. Then for all $k\geq 1$ and all $\lambda\in\C\cup\{\infty\}$ for which $w_k(\Lc;\lambda)$ is finite, it holds
\[
\left|w_k(\Lc;\lambda)-w_k(\Mc;\lambda)\right|\leq 1.
\]
\end{prop}

The following result is the main result of this paper.
It describes the maximal change of the Weyr characteristic of a regular operator pencil
under rank-one perturbations. This result was obtained for matrix pencils in \cite{GernTrunk17} based on \cite[Lemma~2.1]{DopMordTe2008}. 
It is interesting to note that the same result is true for matrices $A$ and $E$,
where $E$ equals the identity, see \cite{Savc03} and also \cite{BehrLebe15,HoerMell94,Thom80}. 
\begin{theorem}
Let $\Ac(\lambda)=\lambda E-A$ satisfies \eqref{torte} and consider $\hat \Ac(\lambda)=\lambda\hat E-\hat A$ given by \eqref{rk_one} with Weyr characteristics  $w(E,A;\lambda)=(w_i(E,A;\lambda))_{i\geq 1}$ and  $w(\hat E,\hat A;\lambda)=(w_i(\hat E,\hat A;\lambda))_{i\geq 1}$. 
If, in addition, $\hat \Ac(\lambda)$ is regular and $w_k(E,A;\lambda)$ is finite for some $k\geq 1$ and $\lambda\in\C\cup\{\infty\}$ then 
	\begin{align}
	\label{1stage}
	\left|w_k(\hat E,\hat A;\lambda)-w_k(E,A;\lambda)\right|\leq 1,\\[1ex]
		\label{trunkungl}
	|\dim \Rf_{\lambda}^k(\hat E,\hat A)-\dim \Rf^k_\lambda(E,A)|\leq k.
	\end{align}
\end{theorem}
\begin{proof}
If $\hat \Ac(\lambda)$ is given by \eqref{rk_one_v}, then $\hat A\hat E^{-1}=(A+wv')(E+uv')^{-1}$ is a one-dimensional perturbation of $AE^{-1}$ by Proposition~\ref{prop:rankoneisonedim}. Note that the regularity of $\lambda E-A$ and $\lambda\hat E -\hat A$ implies with Proposition~\ref{prop:sing} that  $\Rc_c(AE^{-1})=\Rc_c(\hat A\hat E^{-1})=\{0\}$. Invoking Proposition~\ref{prop:francisco}, the Weyr characteristics satisfy for all $k\geq 1$ and all $\lambda\in\C\cup\{\infty\}$
\[
|w_k((A+wv')(E+uv')^{-1};\lambda)-w_k(AE^{-1};\lambda)|\leq 1.
\]
By Theorem~\ref{thm:eq_weyr}, we have 
$$
w_k(AE^{-1};\lambda)=w_k(E,A;\lambda)\quad \mbox{and}\quad   
w_k(\hat A\hat E^{-1};\lambda)=w_k(\hat E,\hat A;\lambda).
$$

If $\hat \Ac(\lambda)$ is given by \eqref{rk_one_u}, then $\hat E|_{\dom A}^{-1}\hat A=(E|_{\dom A}+uv')^{-1}(A+uw')$ is a one-dimensional perturbation of $E|_{\dom A}^{-1}A$ and Proposition~\ref{prop:sing} implies 
$$
\Rc_c(E|_{\dom A}^{-1}A)=\Rc_c(\hat E|_{\dom A}^{-1}\hat A)=\{0\}.
$$
Hence, by Proposition~\ref{prop:francisco} and Theorem~\ref{thm:eq_weyr},
\begin{align*}
&|w_k(E+uv',A+uw';\lambda)-w_k(E,A;\lambda)|\\
&=|w_k((E+uv')^{-1}(A+uw');\lambda)-w_k(E|_{\dom A}^{-1}A;\lambda)|\leq 1.
\end{align*}
The inequality \eqref{trunkungl} follows from \eqref{1stage} and
\begin{align*}
\dim\Rf_{\lambda}^k(E,A)&=\dim\Rf_\lambda^1(E,A)+\ldots+ \dim\frac{\Rf_{\lambda}^k(E,A)}{\Rf_\lambda^{k-1}(E,A)}\\&=w_1(E,A;\lambda)+\ldots+w_k(E,A;\lambda).
\end{align*}
\end{proof}

\bibliographystyle{plain}
{\footnotesize
\bibliography{bibo}
}
\end{document}